\documentclass[12pt]{amsart}
\usepackage{amssymb,latexsym,amsmath,epsfig,amsthm} %% Add other packages as necessary

\newtheorem{theorem}{Theorem}

\newcommand{\invs}{\rule[-.01in]{0in}{.18in}}

\begin{document}

%Next comes your title and author list information in the following construct:

\title{Sum of two repdigits a square} 
\author{Bart Goddard}
\author{Jeremy Rouse}

\address{Department of Mathematics, University of Texas at Austin, Austin,
  TX 78712, USA} 
\email{goddardb@math.utexas.edu} 
\address{Department of Mathematics and Statistics, Wake Forest University, Winston-Salem, NC 27109, USA}
\email{rouseja@wfu.edu}
\subjclass[2010]{Primary 11D61; Secondary 11G05, 11A63}
\begin{abstract}
  A \emph{repdigit} is a natural number greater than 10 which has all
  of its base-10 digits the same.  In this paper we find all examples
  of two repdigits adding to a square.  The proofs lead to interesting
  questions about consecutive quadratic residues and non-residues, and
  provide an elementary application of elliptic curves.
\end{abstract}

\maketitle

\section{Preliminaries and Notation}

At the joint birthday party of Bilbo and Frodo Baggins \cite{Tolkein},
the head table seats 144 which, we are told, is the sum eleventy-one
and thirty-three, the respective ages of the two hobbits.  A
\emph{repdigit} is a positive integer all of whose (base-10) digits
are the same.  The hobbits have provided us with an example of two
repdigits adding to a perfect square: $111+33 = 144$.  In this paper
we find all examples of two repdigits summing to a square.

While Sloan (A010785) allows 1-digit numbers to be repdigits, in this paper
we will not, (since nothing is \emph{rep}eated.)  In this way, we avoid examples
like $2+2=4$ and $7+9=16$.  Also, if we add a repdigit consisting of an
even number of 9's to 1 (a 1-digit repdigit) the result is an even power
of 10, which is a perfect square.  We prefer to avoid these cases and
assume all repdigits are greater than $10$.

For the sake of definiteness, if $a$ is a decimal digit, the repdigit
$aa\ldots a$, consisting of $m$ $a$'s will be denoted $a_m$.  In this
notation, the hobbit example reads $1_3 + 3_2 = 12^2$.  A little algebra
yields the formula:

$$a_m = aa\ldots a = a(11\ldots 1) = a\left(\frac{10^m-1}{9}\right).$$

The main result of the present paper is a complete classification
of perfect squares that are the sum of two repdigits.

\begin{theorem}
\label{main}
The only perfect squares that are sums of two repdigits are the following:
\begin{align*}
  11^{2} &= 99+22 = 88+33 + 77+44 = 66+55\\
  12^{2} &= 111+33\\
  38^{2} &= 1111+333\\
  211^{2} &= 44444 + 77.
\end{align*}
\end{theorem}

In Sections~\ref{mod1} and \ref{mod2} we use congruence conditions to
show that if $a_{m} + b_{n}$ is a perfect square then one of $m$ or
$n$ is $< 6$ and eliminate as many cases as possible. In
Section~\ref{ec}, we handle the remaining cases by using results about
integer points on elliptic curves and conclude the proof of
Theorem~\ref{main}. In Section~\ref{further} we raise some further
questions.

\section{A Couple of Pleasant Properties of Base-10}
\label{mod1}

Of course, once our result is proven in base-10, we would want to
generalize to other bases.  However, $10-1$ is a perfect square, a
 property of $10$ not true of most bases, so the
problem may be much harder in non-decimal bases.  Another happy
accident is that there is a useful string of consecutive quadratic non-residues
modulo $10^6$, which is not to be expected using other bases. This
all comes out in our first theorem, where we show that in the case of
two repdigits adding to a square, at least one of the repdigits has
fewer than 6 digits:

\begin{theorem}
\label{lessthansix}
Let $a$ and $b$ be non-zero base-10 digits.  If $a_m + b_n$ is a
perfect square, with $m \geq n \geq 2$, then $n<6$.
\end{theorem}
\begin{proof}
Suppose that $a_m+b_n = k^2$ for some integer $k$, and
that $m$ and $n$ are both at least 6.  Then using
our formula above, we have

$$ a\left(\frac{10^m-1}{9}\right) + b\left(\frac{10^n-1}{9}\right) = k^2 $$

which simplifies to

\begin{equation}
a10^m +b10^n -(a+b) = (3k)^2.
\end{equation}

Since $m$ and $n \geq 6$, we can reduce this last equation modulo $10^6$
to get

$$-(a+b) \equiv (3k)^2 \pmod{10^6}.$$

Because $a$ and $b$ are non-zero base-10 digits, $-(a+b)$ must be an integer
between $-2$ and $-18$, inclusive.  The congruence above insists that
$-(a+b)$ is a quadratic residue modulo $10^6$.  But inspection of
Table A in the Appendix shows that this is impossible. That is, the
last column in the table consists of all X's, which means that none
of the numbers $-2, -3, -4, \ldots, -18$ is a quadratic residue
modulo $10^6$. We conclude that $n<6$.
\end{proof}

\section{Congruence Considerations}
\label{mod2}

The rest of this paper is devoted to showing that $m$ must also
be less than $6.$  So we now assume that $m\geq 6$ and start
eliminating cases. Our definition of \emph{repdigit} implies that
$m$ and $n$ are both at least $2$, so reducing Equation (1)
modulo $10^2$ gives us

$$-(a+b) \equiv (3k)^2 \pmod{10^2}.$$

Inspection of the second column of Table A shows us that $-(a+b)
\equiv -4, -11$, or $-16 \pmod{10^2}$, which divides our work into
three cases.

\vspace*{.1in}

\noindent{\bf Case 1:} $(a+b) = 4.$

\vspace*{.1in}

In this case the only possibilities are $(a,b) = (3,1), (2,2), \text{ or } (1,3)$.
Also note that, from Table A, $-4$ is a quadratic non-residue modulo $10^4$,
so we must have $n<4$ in this case. (Else, we reduce Equation (1) modulo
$10^4$ and get a contradiction.)  Thus, the only possibilities are:

$$3_m+11, 2_m + 22, 1_m + 33, 3_m +111, 2_m + 222, \mbox{ and } 1_m+333.$$

Four of these six subcases can be eliminated by congruence considerations.
If we put the case $3_m+11$ into equation (1) and reduce modulo $10^6$, we have

$$3\cdot10^m -3 +10^2 - 1 \equiv 96 \equiv (3k)^2 \pmod{10^6}.$$

A solution to this congruence implies $96+d10^6 = x^2$ for some integers $d$
and $x$.  Factoring $2^5$ from the left side gives us $2^5(3+2d \cdot 5^6) = x^2$.
The quantity in parentheses is odd, and so $x^2$ must be exactly divisible
by $2^5$, which is impossible.  In other words, $96$ is not a quadratic
residue of $10^6$, so there is no solution.  In the future, we'll just ask Maple
whether a number is a quadratic residue.

With a similar calculation, the three subcases $1_m+33$, $2_m +222$ and $1_m+333$
require $296$, $1996$ and $2996$, respectively, to be quadratic residues modulo $10^6$.
Maple says that all three of them are quadratic nonresidues, so we have eliminated
these three subcases.  We will handle the remaining cases $2_m+22$ and $3_m+111$
later.

\vspace*{.1in}

\noindent{\bf Case 2:} $(a+b) = 11.$

\vspace*{.1in}

In this case the possibilities are $(a,b) = (2,9)$, $(3,8)$, $(4,7)$, $(5,6)$,
$(6,5)$, $(7,4)$, $(8,3)$, or $(9,2).$   Note that if $n\geq 3$, then modulo $10^3$, equation
(1) reduces to $-11 \equiv (3k)^2 \pmod{10^3}$, and that $-11$ is a quadratic nonresidue
of $10^3.$  Therefore $n=2$ is the only possibility.  This leaves us with 8 subcases:
$2_m+99$, $3_m +88$, $4_m +77$, $5_m+66$, $6_m+55$, $7_m+44$, $8_m+33$, and $9_m+22$.

As we did in Case 1, we put each of these subcases in equation (1) and reduce
modulo $10^6$.  Respectively, these congruences require $889$, $789$, $689$,
$589$, $489$, $389$, $289$ and $189$ to be quadratic residues modulo $10^6$.
Maple says that $789$, $589$, $389$ and $189$ are quadratic non-residues
module $10^6$, so we have eliminated the subcases $3_m+88$, $5_m+66$, $7_m+44$
and $9_m+22$.   The other four subcases we save for Section~\ref{ec}.

\vspace*{.1in}

\noindent{\bf Case 3:} $(a+b) = 16.$

\vspace*{.1in}

In this case the possibilities are $(a,b) = (7,9)$, $(8,8)$ or
$(9,7)$.  The $(8,8)$ subcase would have $8_m + 8_n = k^2$, in which
case $k$ would be even and we could divide the equation by $4$ to get
$2_m + 2_n = (k/2)^2$, which is one of our leftover subcases from Case
1.  When $2_m +22$ is eliminated in Section~\ref{ec}, it will take
this subcase with it.

From Theorem 1, $n\leq 5$, so the possibilites are $9_m+77$, $9_m + 777$,
$9_m+7777$, $9_m+77777$, $7_m+99$, $7_m+999$, $7_m + 9999$, and $7_m + 99999.$
We again put each subcase into equation (1) and reduce modulo $10^6$. Respectively,
$700 - 16$, $7000-16$, $70000 - 16$, $700000 - 16$, $900-16$, $9000-16$,
$90000 - 16$ and $900000-16$ are required to be quadratic residues modulo $10^6$.
But only $700000-16$, $90000 - 16$ and $900000-16$ are.  So we are left
with only $9_m + 77777$, $7_m+9999$, and $7_m+99999$ as possibilities in this case.

It turns out that $90000-16$ and $700000-16$ are not a quadratic residues modulo 7,
so we can also eliminate $9_m+77777$ and $7_m+9999$ from the list.

\section{Elliptic Curve Considerations}
\label{ec}

We have eliminated all cases except this short list:  $2_m+22$, $3_m+111$,
$2_m+99$, $4_m+77$, $6_m+55$, $8_m+33$ and $7_m+99999$.  Each of these cases
will be broken into three subcases depending on the character of $m$ modulo
3.  In most cases, this will eliminate the case.  In the rest, we will discover
the known solutions to the problem.  We will rely on SAGE to find all the
integer points on the elliptic curves which arise.

We'll start with $8_m+33$.  Suppose $m=3l$ for some integer $l$.  Then
equation (1) becomes $8\cdot 10^{3l} + 300 - (8+3) = (3k)^2$.  If we
set $y = 3k$ and $x=2\cdot 10^l$, the equation is $y^2 = x^3 +
289$.
According to Siegel's famous theorem \cite{Siegel}, this curve can
have only finitely many integral points.  SAGE 6.7 implements an
algorithm due to Pethö, Zimmer, Gebel, and Herrmann \cite{PZGH}, which
will find all integral points on an elliptic curve.
% All elliptic curves have genus 1.
SAGE says that the only integer points on this curve are $(-4, 15)$,
$(0,17)$ and $(68, 561)$.  We can observe that none of the $x$-values
are of the form $2\cdot 10^l$, which eliminates this possibility.

Next suppose that $m=3l+1$ for some integer $l$.  Equation (1), for this case,
reads $8\cdot 10^{3l+1} + 289 = (3k)^2.$  Multiply through
by $100$ to get $8\cdot 10^{3l+3} + 28900 = (30k)^2$.  Set $x=2\cdot 10^{l+1}$
and $y=30k$ and we have $y^2 = x^3+28900$.  SAGE says that the only
integer point on this curve is $(0,170)$ and since $0 \neq 2\cdot 10^{l+1}$,
we have eliminated this case.

Next suppose that $m=3l+2$ for some integer $l$.  Equation (1), for this case
reads $8\cdot 10^{3l+2} + 289 = (3k)^2$.  Multiply through by $10^4$ to get
$8\cdot 10^{3l+6}+289\cdot 100^2 = (300k)^2$.  Set $y=300k$ and $x=2\cdot 10^{l+2}$
and we have $y^2 = x^3+2890000$.  SAGE says that the only integer points on
this curve are $(-136,612)$, $(0,1700)$, $(200,3300)$ and $(425,8925)$.  Looking
at the $x$-values, we see one that fits:  $200 = 2\cdot10^{l+2}$.  So $l = 0$
and $m=2$ and we're led to the solution $88+33 = 121$.

\vspace{.1in}

The other cases yield similarly.  If $m=3l+r$, with $r=0,1,$ or $2$, then
we multiply equation (1) by $a^2 10^r$ to get

$$a^310^{3(l+r)} +a^2 10^r(b10^n-(a+b)) = (3a\cdot10^r k)^2.$$

Then we set $x=a10^{l+r}, y= 3a\cdot10^rk$ and $N=a^2 10^r(b10^n-(a+b))$ to
get the elliptic curve

$$y^2 = x^3 + N.$$

Then we ask SAGE to give us all
the integer points on the curve.  If any of those points have $x$-coordinate
of the form $a \cdot 10^p$, then we might have a solution (and in each case,
it will turn out that $m<6$.)  If not, the case
is eliminated. In the case $a=8$, which is already a cube, it is
sufficient to multiply equation (1) by $10^r$, as we did above.
Table B in the Appendix summarizes the calculations, where $r$ is the least
non-negative residue of $m$ modulo $3$.  Whenever the $x$-coordinate of
an integer point matches the form $a10^{l+r}$, we put that number in bold
in the last column of the table.  In each such case, $l=0$ or $1$, which
means that $m<6$.

\begin{proof}[Proof of Theorem~\ref{main}]
Therefore every case of the sum of two repdigits equaling a square must
have repdigits of 5 or fewer digits.  There are only 45 repdigits of
2, 3, 4 or 5 digits, so that leaves us $45\cdot46/2 = 1035$ cases to check.
This is easily done and we find the complete list of solutions as in the
statement.
\end{proof}

\section{Further Observations and Questions}
\label{further}

The proof of Theorem 2 relies entirely on the fact that there is a
convenient string of consecutive quadratic non-residues modulo $10^6$
(due mostly to the fact that $10$ is even.)  Finding strings of
consecutive quadratic residues and non-residues has been studied. (See
\cite{BH} and \cite{Wright}.)  Certainly base-10 is special, but there
are some observations in other bases:

We noted above that there is a family of solutions $22+99 = 33+ 88 =
44+77 = 55+66 = 11^2$ in base 10. In an arbitrary base $c \geq 2$,
we have $(c+1)^2 = (c-k)(c+1) + (k+1)(c+1)$ for $1 \leq k \leq c-1$.
This shows that there is a similar family regardless of the choice of base.

Similarly, $(c+2)^2 = c^2 + 4c + 4 = 1 (c^2+c+1) + 3 (c+1)$ shows that
$111+33 = $ is a square in every base.  Also $(2c^2 + c + 1)^2 =
4c^4 + 4c^3 + 5c^2 + 2c + 1 = 4(c^4+c^3+c^2+c+1) + (c-3)(c+1)$, so
$44444+77$ is just the base 10 manifestation of this identiy.

In the more special case that $c= m^2 - 1$, we have the identity
$c^5 + c^4 + 4c^3 + 4c^2 + 4c + 4 = (c+1)(c^2+2)^2 = m^2 (c^2+2)^2$,
which is to say $(111111)_{c} + (3333)_{c}$ is a square.

The most interesting solutions to our equation seem to come in groups.
$111+33=12^2$ and $1111 + 333 = 38^2$ seem related, as do $4+77 = 9^2$,
$44+77=11^2$ and $44444+77 = 221^2$.  Also the infinite number of solutions
$9_{2m}+1 = (10^m)^2.$  (Here we're allowing 1-digit solutions.)  As
far as we know, this is accidental, but perhaps there's a useful
reason behind this.  If so, maybe there's a solution to the problem
that doesn't rely on SAGE computing on elliptic curves(?)
%Note: There are other solutions to equations of this type without
% using elliptic curves, but knowing a priori that n < 6 is still necessary.
% The prototypical example is the "Ramanujan-Nagell" equation -
% 2^n - 7 = x^2. These can be solved using algebraic number theory in some
% cases.

As mentioned earlier, the case of base 10 is easier than some others.
Of particular note is base 7, where there are a number of examples,
the most striking of which is
\[
  48060^2 = 2309763600 = (11111111111)_{7} + (3333333)_{7}.
\]
It seems difficult to use modular arithmetic to put a bound on $m$ and $n$
subject to
\[
  a \cdot \frac{7^{m} - 1}{6} + b \cdot \frac{7^{n} - 1}{6} = x^{2}.
\]
The solutions to this equation are a subset of those to
\[
  A u^{5} + B v^{5} + Cx^{5} = 6w^{2} x^{3}.
\]
This defines a (possibly singular) surface of general type, and it is
conjectured by Bombieri and Lang that the \emph{rational} points
on a surface of general type lie on a finite union of curves. This supports
our suspicion that there should only be finitely many solutions, but a proof
eludes us.

\pagebreak

\section{Appendix: Tables}

{\bf Table A:  Quadratic residues modulo $10^n$.}  In the following table, generated
with Maple, the entry is O if $-(a+b)$ is a quadratic residue modulo $10^k$ and X otherwise.

\vspace*{.1in}

\begin{tabular}[t]{|c|c|c|c|c|c|}\hline
$-(a+b)$ & $10^2$ & $10^3$ & $10^4$ & $10^5$ & $10^6$ \invs \\ \hline\hline
$-2$ & X & X & X & X & X \invs \\ \hline
$-3$ & X & X & X & X & X \invs \\ \hline
$-4$ & O & O & X & X & X \invs \\ \hline
$-5$ & X & X & X & X & X \invs \\ \hline
$-6$ & X & X & X & X & X \invs \\ \hline
$-7$ & X & X & X & X & X \invs \\ \hline
$-8$ & X & X & X & X & X \invs \\ \hline
$-9$ & X & X & X & X & X \invs \\ \hline
$-10$ & X & X & X & X & X \invs \\ \hline
$-11$ & O & X & X & X & X \invs \\ \hline
$-12$ & X & X & X & X & X \invs \\ \hline
$-13$ & X & X & X & X & X \invs \\ \hline
$-14$ & X & X & X & X & X \invs \\ \hline
$-15$ & X & X & X & X & X \invs \\ \hline
$-16$ & O & O & O & O & X \invs \\ \hline
$-17$ & X & X & X & X & X \invs \\ \hline
$-18$ & X & X & X & X & X \invs \\ \hline

\end{tabular}

\pagebreak

{\bf Table B:  Summary of calculations.}

\vspace*{.1in}

\begin{tabular}[t]{|c|c|c|c|c|c|}\hline
Case  & $r$ & $x$ & $y$ &  $N$ & $x$-coords \invs \\ \hline\hline
$8_m+33$ & 0 & $2\cdot10^l$ & $3k$ & $289$ & $-4, 0, 68 $ \invs \\ \hline
$8_m+33$ & 1 & $2\cdot10^{l+1}$ & $30k$ & $28900$ & $ 0 $ \invs \\ \hline
$8_m+33$ & 2 & $2\cdot10^{l+2}$ & $300k$ & $2890000$ & $ -136, 0, \mathbf{200}, 425 $ \invs \\ \hline
$7_m+99999$ & 0 & $7\cdot10^l$ & $21k$ & $44099216$ & $10577$ \invs \\ \hline
$7_m+99999$ & 1 & $7\cdot10^{l+1} $ & $210k$ & $4409921600$ & $-1064$ \invs \\ \hline
$7_m+99999$ & 2 & $7\cdot10^{l+2} $ & $2100k$ & $440992160000$ & $-5936, -5900, -5516, 2800,$ \invs \\
& & & & & $20825, 21056, 721364000$ \invs \\ \hline
$6_m+55$ & 0 & $6\cdot10^l$ & $18k$ & $17604$ & $-20, -12, 816$ \invs \\ \hline
$6_m+55$ & 1 & $6\cdot10^{l+1} $ & $180k$ & $1760400$ & $-120, 24, 160, 14640$ \invs \\ \hline
$6_m+55$ & 2 & $6\cdot10^{l+2} $ & $1800k$ & $176040000$ & $\mathbf{600}$ \invs \\ \hline
$4_m+77$ & 0 & $4\cdot10^l$ & $12k$ & $11024$ & 1 \invs \\ \hline
$4_m+77$ & 1 & $4\cdot10^{l+1}$ & $120k$ & $1102400$ & $-100, -95, -16, \mathbf{40}, 160,$ \invs \\
& & & & & $584, 1420, 26764 $ \invs \\ \hline
$4_m+77$ & 2 & $4\cdot10^{l+2}$ & $1200k$ & $110240000$ & $-464, -400, 64, \mathbf{400},$ \invs \\
& & & & & $425, 625, 1076, $ \invs \\
& & & & & $\mathbf{4000}, 1154800$ \invs \\ \hline
$3_m+111$ & 0 & $3\cdot10^l$ & $9k$ & $8964$ & $21$\invs \\ \hline
$3_m+111$ & 1 & $3\cdot10^{l+1}$ & $90k$ & $896400$ & $-96, -80, -15, 25, $ \invs \\
& & & & & $40, 49, 120, 256, $ \invs \\
& & & & & $280, 1200, 16576$ \invs \\ \hline
$3_m+111$ & 2 & $3\cdot10^{l+2}$ & $900k$ & $89640000$ & $-375, 124, \mathbf{300},$ \invs \\
& & & & & $700, 5241 $ \invs \\ \hline
$2_m+99$ & 0 & $2\cdot10^l$ & $6k$ & $3556$ & none \invs \\ \hline
$2_m+99$ & 1 & $2\cdot10^{l+1}$ & $60k$ & $355600$ & $ -55, -40, 144, 4320$ \invs \\ \hline
$2_m+99$ & 2 & $2\cdot10^{l+2}$ & $600k$ & $35560000$ & $-216, \mathbf{200}, 704, 800,$ \invs \\
& & & & & $1500, 1800, 7400$ \invs \\
& & & & & $199200$ \invs \\ \hline
$2_m+22$ & 0 & $2\cdot10^l$ & $6k$ & $784$ & $-7, 0, 8, 56$ \invs \\ \hline
$2_m+22$ & 1 & $2\cdot10^{l+1}$ & $60k$ & $78400$ & $-40, 0, 56, 140, 480$ \invs \\ \hline
$2_m+22$ & 2 & $2\cdot10^{l+2}$ & $600k$ & $7840000$ & $-175, 0, 224, 800$ \invs \\ \hline
\end{tabular}

\pagebreak

\end{document}